\documentclass[11pt,a4paper]{article}
  \usepackage{amsmath,amssymb}
  \usepackage{booktabs}
  \usepackage{geometry}
  \newtheorem{theorem}{Theorem}  
  \newtheorem{proposition}[theorem]{Proposition}

  \newtheorem{lemma}[theorem]{Lemma}
  \newtheorem{definition}[theorem]{Definition}
  
  \newtheorem{corollary}[theorem]{Corollary}  

  \newcommand{\connectivity}{connected}
  
  \newcommand{\eop}{\hfill{$\Box$}}
  \newenvironment{proof}
  {\begin{trivlist}\item[]{{\sc Proof.}}}{\eop\noindent\end{trivlist}}

\begin{document}
  \title{Optimal control of the convergence time in the Hegselmann--Krause dynamics}
  \author{{\sc Sascha Kurz}\thanks{sascha.kurz@uni-bayreuth.de}\\ 
      Department of Mathematics, University of Bayreuth\\ 
      D-95440 Bayreuth, Germany}
  \maketitle
  \vspace*{-4mm}
  \noindent
  {
    \center\small{Keywords: opinion dynamics, Hegselmann--Krause model, convergence time, 
    optimal control\\MSC: 39A60, 37N35, 91D10, 93C55\\} 
  }
  \noindent
  \rule{\textwidth}{0.3 mm}

  \begin{abstract}
    \noindent
    We study the optimal control problem of minimizing the convergence time in 
    the discrete Hegselmann--Krause model of opinion dynamics. The underlying 
    model is extended with a set of strategic agents that can freely place their
    opinion at every time step. Indeed, if suitably coordinated, the strategic agents 
    can significantly lower the convergence time of an instance of the 
    Hegselmann--Krause model. We give several lower and upper worst-case bounds 
    for the convergence time of a Hegselmann--Krause system with a given number 
    of strategic agents, while still leaving some gaps for future research.      
  \end{abstract}
  \noindent
  \rule{\textwidth}{0.3 mm}

\section{Introduction}

The dynamics of opinion formation using agent-based models has been studied for more than half a century, see 
e.g.\ \cite{hegselmann2002opinion} for a (partial) overview of different models. Here we consider one specific model, 
which generated a lot of research papers in the social simulation community, the so-called bounded confidence model 
also known as the Hegselmann--Krause model. As originally defined in \cite{Krause:SozDyn:1997}, see also \cite{hegselmann2002opinion}, 
the Hegselmann--Krause model considers discrete time and a finite set of agents with opinions in $\mathbb{R}$, which 
can easily be generalized to $\mathbb{R}^d$ as an opinion space. At each time step the opinion of every agent is updated by averaging certain other opinions. However, an 
agent is influenced only by those opinions that are \textit{near} to his or her own opinion. We will give a precise specification 
of the model in Section~\ref{sec_hk}.

One key issue of a model for opinion dynamics is the question of convergence, see e.g.\ 
\cite{blondel2005convergence,hendrickx2006convergence,lorenz2005stabilization,moreau2005stability,kurz2011hegselmann}. Here by 
convergence we mean a final stable state, where the agents' opinions remain the same in all subsequent time steps. Some authors 
also say that the system is \textit{in equilibrium} or has \textit{frozen}. If convergence is guaranteed, the next question is 
about the necessary number of time steps to reach this state. For the 
Hegselmann--Krause model the convergence time has been upper bounded by $n^{O(n)}$ in \cite{chazelle2011total}. The author's 
conjecture of a polynomial convergence time was, to the best of our knowledge, first proven for the special case of dimension
$d=1$ in \cite{martinez2007synchronous}. Their upper bound $O(n^5)$ was subsequently improved to $O(n^4)$ in 
\cite{touri2011discrete} and to $O(n^3)$ in \cite{bhattacharyya2013convergence,mohajer2012convergence}. A polynomial upper 
bound for a general dimension $d$ was presented in \cite{bhattacharyya2013convergence} and recently improved 
in\cite{martinsson2015improved}. For dimension $d=1$ an $\Omega(n)$ 
worst-case lower bound was given in \cite{martinez2007synchronous}. This was improved recently to $\Omega(n^2)$ in \cite{wedin2014quadratic}. 
For dimension $d=2$ an example yielding a lower bound of $\Omega(n^2)$ was also given in \cite{bhattacharyya2013convergence}.   
Some first few exact values of the worst-case convergence time in dimension $d=1$ were determined in \cite{kurz2014long} using 
integer linear programming techniques.

Recently researchers started to look at the problem of controlling or steering an instance of an opinion dynamics system towards 
a desired state, see \cite{Wongkaew+Caponigro+Borzi:LeadershipFlocking:2014,albi2014kinetic,opinion_control,borzi2014modeling}. 
Exogenous interventions into the systems dynamics were e.g.\ studied in \cite{Mirtabatabaei+Jia+Bullo:OD-BC-Convergence:2014,
Fortunato+Stauffer:OD-Simulations:2006,Fortunato:OD-DamageSpreading:2005}. The \textit{desired state} in \cite{opinion_control} 
is one in which as many as possible of the opinions lie in some specified subset of the opinion space, called a 
\textit{conviction interval}. The practical problem that one can have in mind is that 
of a political or commercial campaign. Via media channels, speeches or personal communications the opinion dynamics 
can be influenced. Formally such a controllable exogenous influence was modeled by introducing \textit{strategic} agents in 
\cite{opinion_control}. Here we will use strategic agents to accelerate the convergence time in the Hegselmann--Krause model, i.e., 
we study the optimal control problem of minimizing the convergence time.

The remaining part of the paper is organized as follows. In Section~\ref{sec_hk} we formally introduce the Hegselmann--Krause 
dynamics with strategic agents and collect some theoretical observations about the system dynamics. The optimal control problem 
is studied in Section~\ref{sec_main}, i.e., we give several lower and upper worst-case bounds for the convergence time of a
Hegselmann--Krause system with a given number of strategic agents. We draw a conclusion and propose some future lines of research 
in Section~\ref{sec_conclusion}.  

\section{The Hegselmann--Krause dynamics with strategic agents}
\label{sec_hk}

As originally defined, the Hegselmann--Krause model considers a one-dimensional continuous opinion space $\mathbb{R}$ \footnote{In 
some papers the unit interval $[0,1]$ or other intervals of the real line are used. For a finite number of agents those variants 
are equivalent via scaling.}, a set $N=\{1,\dots,n\}$ of agents -- these are the non-strategic agents -- and discrete time. At 
each time $t\in\mathbb{N}_{\ge 0}$ each agent $i\in N$ has a certain opinion $x_i(t)\in\mathbb{R}$. The opinion of an agent $i\in N$ 
is influenced at time $t$ only by those agents which have a similar opinion, more precisely, where the distance between the respective 
opinions is at most $\varepsilon$. In our rescaled version of an opinion space $\mathbb{R}$ we can assume w.l.o.g.\ $\varepsilon=1$. 
The new opinion $x_i(t+1)$ of agent $i$ is then determined as the mean of all opinions\footnote{Note that every agent 
influences itself.} that influence agent $i$ at time $t$.

Here we want to extend the model by strategic agents. To this end we denote the set of strategic agents by $S$ and the set of
non-strategic agents by $N$. The union of these sets is denoted by $A=N\cup S$. The opinion $x_s(t)$ of a strategic agent $s\in S$ can be freely chosen at each time $t\in \mathbb{N}_{\ge 0}$ as 
any real number. However, in determining the new opinion at time $t+1$ the non-strategic agents make no difference 
between strategic and non-strategic agents, but are equally influenced by all agents in $A$. To be more precise, we set
\begin{equation}
  \label{eq_hk_dynamics}
  x_i(t+1)=\frac{\sum\limits_{j\in A:\Vert x_i(t)-x_j(t)\Vert\le 1}x_j(t)}{\left|\left\{j\in A:\Vert x_i(t)-x_j(t)\Vert\le 1 \right\}\right|}
\end{equation}
for all $i\in N$. Here, $\Vert\cdot\Vert$ denotes the Euclidean norm (or the absolute value, as it makes no difference 
in dimension $1$). As a model extension we may also use Equation~(\ref{eq_hk_dynamics}) for strategic agents at time steps where 
they do not freely reset their opinion. However, in this paper we assume that each strategic agent freely chooses an opinion at every 
time step.  

We say that two agents $i,j\in A$ are \textit{neighbors} at time $t$ if $\Vert x_i(t)-x_j(t)\Vert\le 1$, which induces an 
\textit{influence graph} $\mathcal{G}_t$. For brevity, we denote the \textit{neighborhood} of an agent $i\in A$ at time $t$ 
by $\mathcal{N}_i(t)=\left\{j\in A\,:\, \Vert x_i(t)-x_j(t)\Vert\le 1\right\}$, i.e., the set of neighbors. With this,  
Equation~(\ref{eq_hk_dynamics}) can be rewritten as $x_i(t+1)=\sum_{j\in\mathcal{N}_i(t)} x_j(t)/\left|\mathcal{N}_i(t)\right|$. 

The Hegselmann--Krause dynamics, with or without strategic agents, is rather complicated and hard to treat analytically, i.e., 
besides convergence not too many theoretical results are known. For the 
special case of dimension $1$ at the very least the ordering of the opinions of the agents is preserved, as observed in several 
papers, see e.g.\ \cite[Lemma 2]{krause2000discrete}:
\begin{proposition}
  \label{prop_ordering}
  If $x_i(0)\le x_j(0)$ for $i,j\in N$, then $x_i(t)\le x_j(t)$ for all all $t\in\mathbb{N}_{\ge 0}$.
\end{proposition}

This well known observation is clearly also true, if strategic agents are present. If $i$ or $j$ is a strategic agent, then 
we can not have such a result, since, by definition, strategic agents can choose their opinion freely. W.l.o.g.\ we assume 
$x_i(0)\le x_j(0)$ for all $i\le j$, $i,j\in N$ in the remaining part of the paper. 

For dimension $1$ and the absence of strategic agents there is another well known theoretical insight. If two agents $i$ and 
$j$ are, at a certain time step $t$, in two different {\connectivity} components of $\mathcal{G}_t$, then this property will be 
preserved for all times $t'>t$. As a consequence, if $\Vert x_{i+1}(t)-x_i(t)\Vert>1$, then $\Vert x_{i+1}(t')-x_i(t')\Vert>1$ 
for all $t'>t$ (using the assumed ordering of the starting opinions). However, this need not be true if at least one strategic 
agent is present. It can indeed be easily shown, see e.g.\ \cite{opinion_control}, that one strategic agent suffices to bring 
any configuration of starting positions to a consensus in a finite number of time steps, i.e., the opinions of all non-strategic 
agents coincide after some rounds. 

For the other direction we remark that it is always possible, given a suitably large number of strategic agents and time 
steps, to move different opinions of any two non-strategic agents as far apart from each other as desired. If albeit $x_i(t)=x_j(t)$, 
where $i,j\in N$, at a certain time $t$, then we have $x_i(t')=x_j(t')$ for all $t'>t$, independently of the precise opinions 
of the strategic agents.  

\begin{definition}
  For each $x\in\mathbb{R}$ and each $t\in\mathbb{N}_{\ge 0}$ we define the \textit{weight} of $x$ at $t$ as
  $w_t(x)=\left|\left\{i\in N\,:\, x_i(t)=x\right\}\right|$. With this, the \textit{weight} of agent $i\in N$ at $t$ 
  is defined as $w_t\!\left(x_i(t)\right)$. 
\end{definition}

\begin{proposition}
  We have $0\le w_t(x)\le n$ and $w_t(x)\in\mathbb{N}_{\ge 0}$ for all $x\in\mathbb{R}$, $t\in\mathbb{N}_{\ge 0}$. 
  For all $t\in\mathbb{N}_{\ge 0}$ 
  and all $i\in N$ we have $w_t\!\left(x_i(t)\right)\le w_{t+1}\!\left(x_i(t)\right)$, i.e., the weight of a non-strategic agent 
  is weakly increasing. 
\end{proposition}

If we use the strategic agents to control the system dynamics, we can ensure that some properties will be preserved if we choose 
the opinions of the strategic agents accordingly. To this end we denote by 
$\mathcal{N}'_i(t)=\left\{j\in N\,:\, \Vert x_i(t)-x_j(t)\Vert\le 1\right\}$ the set of non-strategic neighbors of a non-strategic 
agent~$i$. The graph arising from $\mathcal{G}_t$ by restricting the vertex set to non-strategic agents is denoted by $\mathcal{G}_t'$. 
With this we can state:
\begin{definition}
  We call a non-strategic agent $i\in N$ \textit{frozen} if $\left|\mathcal{N}_i'(t)\right|=w_t\!\left(x_i(t)\right)$. For the 
  vertex set $\mathcal{C}$ of a {\connectivity} component of $\mathcal{G}_t'$ we denote by $l(\mathcal{C},t)$ the agent with the 
  smallest index in $\mathcal{C}$ and by $r(\mathcal{C},t)$ the agent with the largest index in $\mathcal{C}$. The \textit{width} 
  $w(\mathcal{C},t)$ of $\mathcal{C}$ at $t$ is given by $x_{r(\mathcal{C},t)}(t)-x_{l(\mathcal{C},t)}(t)\in\mathbb{R}_{\ge 0}$.
\end{definition}

Due to our assumed ordering of the starting opinions, we have $x_{l(\mathcal{C},t)}(t)\le x_i(t)\le x_{r(\mathcal{C},t)}(t)$ for 
all $i\in\mathcal{C}$, i.e., $l(\mathcal{C},t)$ and $r(\mathcal{C},t)$ mark the \textit{ends} of the interval of {\connectivity} 
component $\mathcal{C}$ and $w(\mathcal{C},t)$ denotes the corresponding length. We remark that $w(\mathcal{C},t)=0$ if and only 
if all non-strategic agents of $\mathcal{C}$ are frozen at time $t$. 

\begin{proposition}
  \label{prop_coincidence}
  Given two non-strategic agents $i,j\in N$ and a time $t\in\mathbb{N}_{\ge 0}$ we have $x_i(t+1)=x_j(t+1)$ if and only if 
  $\mathcal{N}_i(t)=\mathcal{N}_j(t)$.
\end{proposition}

If the influence graph $\mathcal{G}_t'$ at time $t$ decomposes into 
the {\connectivity} components $\mathcal{C}_1,\dots,\mathcal{C}_h$, then we denote the \textit{width} of the entire configuration, i.e., 
the sum over the widths of the {\connectivity} components, by $w\!\left(\mathcal{G}_t',t\right)=\sum_{i=1}^h w(\mathcal{C}_i,t)$.

The width of a {\connectivity} component can only increase if a strategic agent is placed behind the ends of the corresponding interval 
but still within the influence range.

\begin{proposition}
  \label{prop_lr_shrinking}
  Let $\mathcal{C}$ be a {\connectivity} component of the influence graph $\mathcal{G}_t'$ at a certain time $t\in\mathbb{N}_{\ge 0}$. 
  If we have $x_a(t)\ge x_{l(\mathcal{C},t)}(t)$ and $x_a(t)\le x_{r(\mathcal{C},t)}(t)$ for all $i\in \mathcal{C}$ and all 
  $a\in\mathcal{N}_i(t)$, then we have $x_{l(\mathcal{C},t)}(t)\le x_j(t+1)\le x_{r(\mathcal{C},t)}(t)$ for all 
  $j\in\mathcal{C}$.
\end{proposition}

So, one can place the opinions of the strategic agents such that the width of the entire configuration does not increase.

Since the range of influence is at most $1$, each non-strategic agent can not move too far within one time step.
\begin{proposition}
  \label{prop_bounded_movement}
  For each $i\in N$ and each $t\in\mathbb{N}_{\ge 0}$ we have $\left\Vert x_i(t)-x_i(t+1) \right\Vert\le 1$.
\end{proposition}

We say that the system has \textit{converged} at time $t$ if all non-strategic agents are frozen, i.e., if we have for every 
$i,j\in N$ either $x_i(t)=x_j(t)$ or $\Vert x_i(t)-x_j(t)\Vert>1$. Without the influence of any strategic agents, we may simply 
assume that they all place their opinions far apart or at one of the positions $x_i(t)$, the opinions of the non-strategic 
agents will remain unchanged, i.e., stable for all times $t'>t$. However, the strategic agents may convert frozen agents to 
non-frozen ones. This can make sense, if one wants to end up with a consensus. Nevertheless, we define the \textit{convergence time} 
of a Hegselmann--Krause system (HK system for short) as the smallest time $t$ such that the system has converged at time $t$. Obviously this definition of convergence 
time depends on the starting opinions $x_i(0)$ of the non-strategic agents and all opinions $x_j(t)$, or rules to compute them, of 
the strategic agents. As already observed in the introduction, the convergence time of a HK system with $n$ non-strategic agents and 
no strategic agent is upper bounded by $O(n^3)$, while the slowest known sequence of examples reaches $\Omega(n^2)$. In the following 
section we will consider the optimization problem of lowering the convergence time using strategic agents, i.e., we consider 
an optimal control problem.  

\section{Lowering the convergence time using strategic agents}
\label{sec_main}

Given a HK system we ask how to optimally coordinate the opinions of the $m=|S|$ strategic agents so that the convergence time
is minimized, i.e., we consider an optimal control problem. We call the corresponding time the \textit{optimal convergence time}, 
given the starting positions of the non-strategic agents and the number $m$ of strategic agents. Mostly we will be interested 
in general assertions rather than considerations of specific examples. So, by $f(n,m)$ we denote the supremum of the optimal 
convergence time over all HK instances with $n$ non-strategic agents and $m$ strategic agents. In this notation our current knowledge 
on the convergence time can be written as $f(n,0)\in O(n^3)$ and $f(n,0)\in\Omega(n^2)$. Obviously, we have 
$f(n,m)\le f(n',m)$ and $f(n,m)\ge f(n,m')$ for all $n'\ge n$ and all $m'\ge m$. 

We start by considering the case of $m=1$ strategic agent and mimic the proof strategy for the best known upper bound of the 
convergence time from \cite{bhattacharyya2013convergence}.  

\begin{lemma}
  \label{lemma_improvement}
  Given a HK system with $m=1$ strategic agent, the opinion of this agent at time $t$ can be chosen in such a way such that
  either 
  \begin{itemize}
    \item[(1)] the weight of a non-strategic agent increases and the width $w(\mathcal{G}_t',t)$ does not increase or
    \item[(2)] the width $w(\mathcal{G}_t',t)$ decreases by at least $\frac{1}{n+1}$
  \end{itemize} 
  at time $t+1$, whenever the HK system has not converged at time $t$.  
\end{lemma}
\begin{proof}
  Let us denote the {\connectivity} components of $\mathcal{G}_t'$ by $\mathcal{C}_1,\dots,\mathcal{C}_h$. If there exists an index 
  $1\le g\le h$ with $0<w\!\left(\mathcal{C}_g,t\right)\le 1$, then we can place the opinion of the strategic agent far away from the
  other opinions so that no agent is influenced. Since $w\!\left(\mathcal{C}_g,t\right)\le 1$ we have 
  $\Vert x_i(t)-x_j(t)\Vert\le 1$, so that $x_i(t+1)=x_j(t+1)$ for all $i,j\in\mathcal{C}_g$. As $w\!\left(\mathcal{C}_g,t\right)>0$ 
  the weights of the agents in $\mathcal{C}_g$ increase by at least $1$ each in this case. Obviously, the width of the influence graph
  does not increase. 
  
  Otherwise we choose an index $1\le g\le h$ with $w\!\left(\mathcal{C}_g,t\right)>1$ and 
  set $x_s(t)=x_{l(\mathcal{C}_g,t)}(t)+1$, where $s$ denotes the strategic agent. With this we have
  \begin{eqnarray*}
    x_{l(\mathcal{C}_g,t)}(t+1)=\frac{1}{\left|\mathcal{C}_g\right|+1}\cdot \left(x_s(t)+\sum_{i\in\mathcal{C}_g}x_i(t)\right)
    \ge x_{l(\mathcal{C}_g,t)}(t)+\frac{1}{n+1},
  \end{eqnarray*} 
  since $x_i(t)\ge x_{l(\mathcal{C}_g,t)}(t)$ for all $i\in\mathcal{C}_g$ and $\left|\mathcal{C}_g\right|\le n$. Using 
  Proposition~\ref{prop_lr_shrinking} we conclude $w(\mathcal{C}_g,t+1)\le w(\mathcal{C}_g,t)-\frac{1}{n+1}$. The widths of the 
  other {\connectivity} components do not increase. It may happen that a certain {\connectivity} component, e.g.\ $\mathcal{C}_g$ 
  itself, decomposes into several components in the considered time step, i.e., $\mathcal{C}_g$ may not be a 
  {\connectivity} component at time $t+1$. Nevertheless, the summed widths of the respective components 
  is not larger than it was originally, while in one component of $\mathcal{C}_g$ a contraction of at least $\frac{1}{n+1}$ occurs. Thus, we 
  have $w(\mathcal{G}_{t+1}',t+1)\le w(\mathcal{G}_t',t)-\frac{1}{n+1}$.      
\end{proof}

\begin{corollary}
  $f(n,1)\in O(n^2)$.
\end{corollary}
\begin{proof}
  Since the weight of every agent is a non-negative integer being at most equal to $n$ and it is weakly increasing over time, 
  Case~(1) of Lemma~\ref{lemma_improvement} can occur at most $n^2$ times. (A more refined analysis would yield that we need to 
  consider this case at most $n-1$ times.) Since the maximum difference between two neighboring agents is $1$, we have 
  $w(\mathcal{G}_0',0)\le n-1$. Thus, Case~(2) of Lemma~\ref{lemma_improvement} can occur at most $(n-1)\cdot(n+1)\le n^2$ times.
\end{proof}

Using one strategic agent the upper bound for the (optimal) convergence time could be improved from $O(n^3)$ to $O(n^2)$, but still 
we do not know whether the strategic agent was necessary, since $f(n,0)\in O(n^2)$ may also be true. To see that already one 
strategic agent can significantly decrease the (optimal) convergence time we consider a specific parametric example:

\begin{lemma}
  \label{lemma_equidistant}
  For $n\in\mathbb{N}_{>0}$ consider the HK system with starting opinions $x_i(0)=i-1$ for all $1\le i\le n$, i.e., the \textit{equidistant} 
  configuration. The convergence time of this example is in $\Omega(n)$ while its optimal convergence time for $m=1$ strategic 
  agent is in $O\!\left(n^{3/4}\right)$.
\end{lemma}
\begin{proof}
  Via induction one can easily show that $x_i(t)=i-1$ for all $t+1\le i\le n-t$, i.e., the movement starts from the two ends
  of the chain of neighbors and concerns two additional agents at each time step. Thus the convergence time is at least $\frac{n-2}{2}$. 
  Actually, the convergence time of the equidistant configuration is given by $\frac{5n}{6}+O(1)$. An exact formula was stated 
  in \cite{kurz2014long} without a proof and rigorously proven in \cite{hegarty2014hegselmann}. 
  
  For the other direction we assume w.l.o.g.\ $n\ge 3^4=81$ and set $k=\left\lfloor n^{1/4}\right\rfloor$. At the beginning 
  all $n$ non-strategic agents are in a single {\connectivity} component and each neighbor of an agent is distance $1$ apart. We 
  inductively will cut off groups of $k$ agents. 
  
  At step 1 there are $n$ remaining agents with distances $1,\dots,1$. Let $x_1$ be the position of the opinion of the $k$th 
  non-strategic agent at time $0$. Setting the opinion of the strategic agent $s$ to $x_1-1$ at time $0$, the first $k$ non-strategic 
  agents are separated from the last $n-k$ ones at time $1$. For the remaining steps we ignore the first $k$ agents.
  
  At time $1$ the distances between the remaining $n-k$ agents are given by $1,\dots,1,\frac{1}{2}$. For step 2 we consider 
  the opinion $x_2$ of the $k$th non-strategic agent counted from the right at time $1$. Setting the opinion of the strategic 
  agent to $x_2+1$ at time $1$, the last $k$ non-strategic agents are separated from the first $n-2k$ ones at time $2$. For the 
  remaining steps we ignore the last $k$ agents.

  At time $2$ we have to consider $n-2k$ agents with distances $\frac{1}{2},1,\dots,1$. So, up to symmetry we are in the same 
  situation as in the previous step. We iterate the separation process using the strategic agent until we end up with at 
  most $\left\lceil n^{3/4}\right\rceil +2$ groups of cardinality at most $k$ after at most $\left\lceil n^{3/4}\right\rceil +2$ 
  time steps.
  
  The dynamics of each of these groups can be considered separately. After the separation process has finished we place the opinion
  of the strategic agent far away from all other opinions so that it does not play a role. Since any HK system with at most $k$ 
  non-strategic agents converges in $O(k^3)$ time, the overall convergence time is in $O\!\left(n^{3/4}\right)$.       
\end{proof}

Next we consider the example from \cite{wedin2014quadratic} yielding the quadratic lower bound bound $f(n,0)\in\Omega(n^2)$:
\begin{lemma}
  For an integer $k\ge 10$ we consider the so-called \textit{dumbbell} configuration with $3k+1$ non-strategic agents, where
  \begin{enumerate}
    \item[(1)] $k$ agents have starting opinion $-\frac{1}{k}$,
    \item[(2)] one agent has starting opinion $i$ for each $0\le i\le k$, and 
    \item[(3)] $k$ agents have starting opinion $k+\frac{1}{k}$.
  \end{enumerate}
  The convergence time of this example is in $\Omega(n^2)$ while its optimal convergence time for $m=1$ strategic 
  agent is in $O\!\left(n^{3/4}\right)$.
\end{lemma}
\begin{proof}
  The convergence time of the dumbbell configuration has been already treated in \cite{wedin2014quadratic}, so that we only 
  consider the optimal convergence time. 
  
  At time $t=0$ we place the strategic agent at $2$. Using the usual ordered numbering of the non-strategic agents, we obtain
  \begin{itemize}
    \item $x_i(1)=-\frac{1}{k+1}$ for all $1\le i\le k$, $x_{k+1}(1)=0$,
    \item $x_{k+2}(1)=\frac{5}{4}$, $x_{k+3}(1)=2$, $x_{k+4}(1)=\frac{11}{4}$,
    \item $x_{k+1+i}(1)=i$ for all $4\le i\le k$, and $x_j(1)=k+\frac{1}{k+1}$ for all $2k+2\le j\le 3k+1$.
  \end{itemize}
  At time $t=1$ we place the strategic agent at $k-2$. With this, we obtain
  \begin{itemize}
    \item $x_i(2)=-\frac{k}{(k+1)^2}$ for all $1\le i\le k+1$, 
    \item $x_{k+2}(2)=\frac{13}{8}$, $x_{k+3}(2)=2$, $x_{k+4}(2)=\frac{19}{8}$,
    \item $x_{k+5}(2)=\frac{9}{2}$, $x_{k+1+i}=i$ for all $5\le i\le k-4$,
    \item $x_{2k-2}(2)=k-\frac{11}{4}$, $x_{2k-1}(2)=k-2$, $x_{2k}(2)=k-\frac{5}{4}$,
    \item $x_{2k+1}(2)=k-\frac{1}{(k+1)(k+2)}$, and $x_{2k+1+i}(2)=k+\frac{k}{(k+1)^2}$ for all $1\le i\le k$.   
  \end{itemize}
  At time $t=3$ the agents $k+2$, $k+3$, and $k+4$ will converge to the joint opinion $2$. Similarly, at time $t=4$ the
  agents $2k-2$, $2k-1$, and $2k$ will converge to the joint opinion $k-2$. The last $k+1$ agents will converge to a joint
  opinion at time $t=3$. The remaining agents from $k+5$ to $2k-3$ form a chain of (almost) equal distances, i.e., with a 
  single exception all distances are equal to $1$, while the exceptional distance is equal to $\frac{1}{2}$ -- a case that has 
  already been considered in the proof of Lemma~\ref{lemma_equidistant}. So, reusing the corresponding reasoning, we 
  conclude a convergence time of $O\!\left(n^{3/4}\right)$.  
\end{proof}

So, we have seen that the control of a single strategic agent can accelerate the convergence time of the equidistant configuration 
by at least $\Omega\!\left(n^{1/4}\right)$. Improving the upper bound for the convergence time $f(n,0)$, i.e., in the absence of strategic 
agents, would increase this gap even more. Besides a tighter analysis, an improved strategy for the strategic agent is also conceivable.
For the dumbbell configuration the demonstrated acceleration is at least of order $\Omega\!\left(n^{5/4}\right)$. However, the 
\textit{power} of a single strategic agent alone is limited, as we will see in the next two results.

\begin{lemma}
  \label{lemma_lower_bound_1_single_component}
  For each integer $k\ge 15$ consider the HK system given by
  \begin{enumerate}
    \item[(a)] $k^2$ non-strategic agents with starting opinion $-\frac{2}{3}$,
    \item[(b)] $k$ non-strategic agents with starting opinion $0$,  
    \item[(c)] $k^2$ non-strategic agents with starting opinion $\frac{2}{3}$, and
    \item[(d)] $m=1$ strategic agent.
  \end{enumerate}
  Let $x_1(t)$ denote the opinion of the agents in (a), $x_2(t)$ denote the opinion of the agents in (b), and 
  $x_3(t)$ denote the opinion of the agents in (c) at time $t\in\mathbb{N}_{\ge 0}$. Then, for all $0\le t\le\frac{k}{8}$ we
  have
  \begin{enumerate}
    \item[(1)] $-\frac{2}{3}-\frac{t}{k}\le x_1(t)\le -\frac{2}{3}+\frac{t}{k}$,
    \item[(2)] $\frac{2}{3}-\frac{t}{k}\le x_3(t)\le \frac{2}{3}+\frac{t}{k}$,
    \item[(3)] $-\left(\frac{t}{k}+\frac{t^2}{k^2}\right) \le x_2(t)\le\left(\frac{t}{k}+\frac{t^2}{k^2}\right)$, 
    \item[(4)] $x_2(t)-x_1(t)\le 1-\frac{1}{k}$,
    \item[(5)] $x_3(t)-x_2(t)\le 1-\frac{1}{k}$, and
    \item[(6)] $x_3(t)-x_1(t)> 1$
  \end{enumerate}
  independently of the precise opinions of the strategic agent.   
\end{lemma}
\begin{proof}
  We remark that inequalities (4)-(6) say that the agents of type (a) influence agents of types (a) and (b), agents of type (b) 
  influence all non-strategic agents, and agents of type (c) influence agents of types (b) and (c). The strategic agent may influence 
  agents of none, some or all types.
  
  We prove by induction on $t$ and observe that for $t=0$ all statements are valid. Now let $t\in \mathbb{N}_{\ge 0}$, 
  with $t\le\frac{k}{8}-1$, and we assume that all six inequalities are valid for $t$.
  
  Since the strategic agent can pull with a force of at most $1$, we have
  $$
    x_1(t+1)\ge x_1(t)+\frac{-1+k^2\cdot 0+k\cdot 0}{1+k^2+k}\ge x_1(t)-\frac{1}{k}\overset{\text{(1)}}{\ge} -\frac{2}{3}-\frac{t+1}{k}.
  $$ 
  Using Inequality~(4) at time $t$ we conclude
  $$
    x_1(t+1)\le x_1(t)+\frac{1+k^2\cdot 0+k\cdot\left(1-\frac{1}{k}\right)}{1+k^2+k}\le  x_1(t)+\frac{1}{k}\overset{\text{(1)}}{\le} -\frac{2}{3}+\frac{t+1}{k}.
  $$
  So, Inequality~(1) is also valid at time $t+1$. The validity of Inequality~(2) is proven analogously. 
  
  In order to prove Inequality~(3) we ask how left the opinions of agents of type (b) can be at time $t+1$. The extreme situation 
  occurs if the strategic agent pulls with force $1$ to the left and the agents of types (a), (b), and (c) are located 
  as left as possible. So, we obtain
  \begin{eqnarray*}
    x_2(t+1)&\ge& \frac{-\left(1+\frac{t}{k}+\frac{t^2}{k^2}\right)+k^2\cdot\left(-\frac{2}{3}-\frac{t}{k}\right)+k\cdot\left(0-\frac{t}{k}-\frac{t^2}{k^2}\right)+k^2\cdot\left(\frac{2}{3}-\frac{t}{k}\right)}{1+k^2+k+k^2}\\
    &=& -\frac{2tk+t+1}{2k^2+k+1} -\frac{\frac{t}{k}+\frac{t^2}{k}+\frac{t^2}{k^2}}{2k^2+k+1}\\
    &\ge& -\frac{t+1}{k} -\frac{(t+1)^2}{k^2}. 
  \end{eqnarray*}
  Analogously, we conclude $x_2(t+1)\le \frac{t+1}{k} +\frac{(t+1)^2}{k^2}$ so that Inequality~(3) is also valid at time $t+1$.
  
  Setting $t'=t+1$ and using inequalities (1) and (3) at time $t'$, we have $x_2(t')-x_1(t')\le \frac{2}{3}+\frac{2t'}{k}+\frac{t'^2}{k^2}$. 
  For all $0\le t'\le -k+\frac{1}{3}\cdot\sqrt{12k^2-9k}$ the right hand side is at most $1-\frac{1}{k}$. Since $k\ge 15$ and $t'\le \frac{k}{8}$ we have 
  $t'\le \frac{k}{8}\le -k+\frac{1}{3}\cdot\sqrt{12k^2-9k}$, so that Inequality~(4) is valid for all $t\le \frac{k}{8}-1$. Analogously, we
  conclude the validity of Inequality~(5).
  
  Setting $t'=t+1$ and using inequalities (1) and (2), we have $x_3(t')-x_1(t')\ge \frac{4}{3}-\frac{2t'}{k}\ge \frac{13}{12}>1$, 
  so that Inequality~(6) is valid for all $t\le \frac{k}{8}$.   
\end{proof}

\begin{corollary}
  \label{cor_lower_bound}
  $f(n,1)\in \Omega(\sqrt{n})$.
\end{corollary} 

We remark that the lower bound from Corollary~\ref{cor_lower_bound} can be increased to $f(n,1)\in \Omega\!\left(n^{2/3}\right)$, see 
Lemma~\ref{lemma_lower_bound_alpha}. However, the result from Lemma~\ref{lemma_lower_bound_1_single_component} has the 
advantage that it uses a connected starting configuration (instead of many separated {\connectivity} components) and does not rely 
on the convergence analysis of the dumbbell configuration.  

If the number of strategic agents is sufficiently increased, with respect to the number of non-strategic agents, then 
the optimal convergence time can be decreased down to a constant.

\begin{theorem}
  \label{thm_brute_force}
  For each $n\in\mathbb{N}_{>0}$ we have $f(n,9n)\le 2$ and $f(n,9n)\in\Theta(1)$.
\end{theorem}
\begin{proof}
  At first we determine a set of positions where we place the opinions of the strategic agents in the first round. To this end, let 
  us denote the {\connectivity} components of $\mathcal{G}_0'$ by $\mathcal{C}_1,\dots,\mathcal{C}_h$. For each index $1\le i\le h$ 
  we proceed as follows: We set $w_i=\left\lceil w\!\left(\mathcal{C}_i\right)\right\rceil$, i.e., $w_i$ is a non-negative integer 
  such that $w\!\left(\mathcal{C}_i\right)\in (w_i-1,w_i]$. For each {\connectivity} component $\mathcal{C}_i$ we will use 
  $k_i$ out of the $9n$ strategic agents in the first round.
  
  If $w_i\le 1$ we place strategic agents for $\mathcal{C}_i$ at 
  $k_i=0$, i.e., no, positions. To ease the subsequent notation we set $p_0^i=\frac{1}{2}\cdot\left(x_{l(\mathcal{C}_i,0)}(0)+
  x_{r(\mathcal{C}_i,0)}(0)\right)$, i.e., we choose the center of the corresponding interval of opinions. 
  
  Now assume $w_i\ge 2$. If $w_i$ is even, we place strategic agents at the $k_i=w_i/2$ positions
  $p_j^i=x_{l(\mathcal{C}_i,0)}(0)+1+\left(2+\varepsilon_i\right)\cdot j$, where $0\le j<k_i$, $j\in\mathbb{N}_{\ge 0}$. Here we choose $\varepsilon_i>0$ 
  suitably small such that $p_{k_i-1}^i\le x_{r(\mathcal{C}_i,0)}(0)$ and the open intervals $(p_j^i+1,p_{j+1}^i-1)$ do not 
  contain opinions of non-strategic agents at time $0$ for $0\le j<k_i-1$. If $w_i$ is odd, we place strategic agents at the 
  $k_i=(w_i+1)/2$ positions $p_j^i=x_{l(\mathcal{C}_i,0)}(0)+\left(2+\varepsilon_i\right)\cdot j$, where $0\le j<k_i$. 
  Here we again choose $\varepsilon_i$ suitably small such that $p_{k_i-1}^i\le x_{r(\mathcal{C}_i,0)}(0)$ and the open 
  intervals $(p_j^i+1,p_{j+1}^i-1)$ do not contain opinions of non-strategic agents at time $0$ for $0\le j<k_i-1$.  
  
  With this we have $x_{l(\mathcal{C}_i,0)}(0)\le p_j^i\le x_{r(\mathcal{C}_i,0)}(0)$ 
  for all $0\le j<k_i$, i.e., strategic agents at these positions influence only agents from $\mathcal{C}_i$. Furthermore, if 
  $k_i>0$, each agent in $\mathcal{C}_i$ is influenced by strategic agents from exactly one position $p_j^i$. It remains to 
  determine the number of strategic agents that should be placed at position $p_j^i$.
  
  Let $b$ be the number of non-strategic agents with a starting opinion which is at most $1$ apart from $p_j^i$. By $a$ we denote 
  the number of non-strategic agents with starting opinion in $\left[p_j^i-2,p_j^i-1\right)$. Similarly, by $c$ we denote 
  the number of non-strategic agents with starting opinion in $\left(p_j^i+1,p_j^i+2\right]$. We place exactly $3\cdot(a+b+c)$ 
  strategic agents at position $p_j^i$ at time $0$. Let $g$ be an arbitrary non-strategic
  agent with starting opinion in $\left[p_j^i-1,p_j^i\right]$ and $\delta=p_j^i-x_g(0)$. With this we have
  $$
    x_g(1)\ge \frac{a_1\cdot \left(p_j^i-1-\delta\right)+b_1\cdot \left(p_j^i-\delta\right)+3\cdot(a+b+c)\cdot p_j^i}{a_1+b_1+3\cdot(a+b+c)},  
  $$   
  where $a_1\le a$ and $b_1\le b$. Since $\delta\le 1$ and $a_1+b_1\le a+b+c$, we have
  $$
    x_g(1)\ge p_j^i -\frac{2\cdot\left(a_1+b_1\right)}{\left(a_1+b_1\right)+3(a+b+c)}\ge p_j^i-\frac{1}{2}.
  $$
  For the other direction we have
  $$
    x_g(1)\le \frac{1\cdot \left(p_j^i-\delta\right)+3\cdot(a+b+c)\cdot p_j^i+b_2\cdot\left(p_j^i+1-\delta\right)}{1+3\cdot(a+b+c)+b_2},
  $$
  where $b_2\le b-1$. Since $\delta\ge 0$ and $b_2+1\le a+b+c$, we have
  $$
    x_g(1)\le p_j^i +\frac{b_2+1}{\left(b_2+1\right)+3\cdot(a+b+c)}\le p_j^i+\frac{1}{2}.
  $$ 
  Similarly, we conclude $p_j^i-\frac{1}{2}\le x_g(1)\le p_j^i+\frac{1}{2}$ for every non-strategic agent $g$ with 
  starting opinion in $\left[p_j^i,p_j^i+1\right]$. 
  
  If $k_i=0$, i.e., when we place no strategic agents for $\mathcal{C}_i$, then we also have 
  $p_0^i-\frac{1}{2}\le x_g(1)\le p_0^i+\frac{1}{2}$ for every non-strategic agent $g\in\mathcal{C}_i$.
  
  When determining the parameters $a$, $b$, and $c$ for the strategic agents at the positions $p_j^i$ every non-strategic 
  agent is counted at most thrice, so that the number of placed strategic agents at time $0$ is at most $3\cdot 3\cdot n=9n$. 
  We place the remaining strategic agents, if there are any, far away from all other opinions, so that they do not influence 
  non-strategic agents.
  
  At time $1$ the influence graph $\mathcal{G}_1'$ consists of unions of complete graphs, i.e., for every non-strategic agent $i$ 
  and every $j\in\mathcal{N}_i'(1)$ we have $\mathcal{N}_i'(1)=\mathcal{N}_j'(1)$, since the opinions of the non-strategic agents
  are clustered into intervals of length at most $1$ at time $1$.       
  
  No influence of strategic agents is needed to cause convergence at time $2$, so that we place the $9n$ strategic agents far away from 
  the opinions of the non-strategic agents. Thus, we have $f(n,9n)\le 2$. For $n\ge 2$ we can consider the example with starting 
  positions given by $i$ for all $0\le i\le n-1$, which has not converged at time $0$, so that $f(n,9n)\ge 1$ for $n\ge 2$ and 
  $f(n,9n)\in\Theta(1)$.  
\end{proof}

We remark that essentially we have used the capabilities of external control in the proof of Theorem~\ref{thm_brute_force} in a 
single round only. Theorem~\ref{thm_brute_force} states that $9n$ strategic agents are sufficient to control any HK system 
in such a way that it converges in a constant number of time steps. For $n\ge 5$ the convergence is, in general, as fast as possible, 
i.e.\ more strategic agents will not help to accelerate the convergence.

\begin{lemma}
  \label{lemma_not_too_fast}
  For $n\ge 5$ consider the example where the starting opinions of the non-strategic agents are given by 
  $x_1(0)=0$, $x_2(0)=\frac{1}{4}$, $x_3(0)=\frac{2}{3}$, $x_4(0)=\frac{5}{4}$, $x_5(0)=\frac{4}{3}$, and $x_i(0)=6$ 
  for $6\le i\le n$. Independently of the number $m$ of strategic agents and their precise positions, the optimal convergence 
  time for this example is at least $2$.
\end{lemma}
\begin{proof}
  Since $x_i(0)\in\left[0,\frac{4}{3}\right]$ for all $i\in \{1,2,3,4,5\}$, we can use Proposition~\ref{prop_bounded_movement} to 
  conclude $x_i(1)\in\left[-1,\frac{7}{3}\right]$ for all $i\le 5$ and $x_i(1)\ge 5$ for all $i>5$. Due to 
  Proposition~\ref{prop_coincidence} all non-strategic agents $i\in\{1,2,3,4,5\}$ have different 
  opinions at time $1$. Thus, by the pigeonhole principle, there exist two non-strategic agents $i,j\in \{1,2,3,4,5\}$ with
  $0<\Vert x_i(1)-x_j(1)\Vert\le 1$, i.e., the system has not converged at time $1$.
\end{proof}

\begin{lemma}
  For each $c_1,c_2\in\mathbb{R}_{>0}$ and each $\alpha\in [0,1)$ we have $f\!\left(n,c_1\cdot n^\alpha\right)\ge c_2$ for all 
  sufficiently large $n\in\mathbb{N}_{>0}$.
\end{lemma}
\begin{proof}
  We consider the HK system with $\left\lfloor\frac{n}{k}\right\rfloor$ equidistant configurations consisting of $k$ 
  non-strategic agents, where $k=2c_2+2$, each. The remaining $n-\left\lfloor\frac{n}{k}\right\rfloor\cdot k\ge 0$ 
  non-strategic agents are placed far away from the equidistant configuration. If one of the {\connectivity} components of 
  $\mathcal{G}_0'$, corresponding to an equidistant configurations, is never influenced by any 
  strategic agent, then it takes at least $c_2$ time steps till convergence. Up to time $c_2-1$ at most 
  $c_2\cdot c_1\cdot n^\alpha$ {\connectivity} components of $\mathcal{G}_0'$ could be influenced by at least one strategic 
  agent. Thus, for $n$ sufficiently large, $c_2$ time steps are not enough so that all initial {\connectivity} components 
  can be affected at least once. 
\end{proof}

So far we have seen, that using a single strategic agent the optimal convergence time is upper bounded by $O(n^2)$ and using 
$9n$ strategic agents the optimal convergence time is upper bounded by the constant $2$. Next we deal with the case in between these 
two extremes.

\begin{theorem}
  \label{thm_in_between}
  For each $\alpha\in[0,1]$ we have $f(n,n^\alpha+12)\in O\!\left(n^{2-2\alpha}\right)$.
\end{theorem}
\begin{proof}
  As usual, we assume that the non-strategic agents are ordered with respect to their starting opinions, i.e., 
  $x_1(0)\le\dots\le x_n(0)$. In a first step we use our strategic agents to influence each non-strategic agent at most once 
  in order to guarantee that after this step the {\connectivity} components have either a \textit{low width} or a \textit{high density}, 
  i.e., many members compared to the corresponding width. We will formulate our reasoning in the style of an algorithm, which we 
  then analyze later on.

  \medskip

  \noindent
  \textbf{Step (1):}\\  
  (1.1) $t\leftarrow 0$, $h\leftarrow 1$, $u\leftarrow n^\alpha+12$, unmark all $i\in N$\\
  (1.2) let $\mathcal{C}$ be the {\connectivity} component in $\mathcal{G}_t'$ with $h\in\mathcal{C}$\\
  (1.3) if $x_h(t)+4\ge x_{r(\mathcal{C},t)}(t)$ then\\
  \hspace*{12mm} $h\leftarrow 1+\max\{i\,:\,i\in\mathcal{C}\}$\\
  \hspace*{12mm} if $h>n$ then STOP else go to (1.2) end if\\
  \hspace*{8mm} end if\\
  (1.4) $k\leftarrow \left|\left\{i\in\mathcal{C}\,:\, x_h(t)\le x_i(t)\le x_h(t)+4\right\}\right|$\\
  (1.5) if $k>\left(n^\alpha+12\right)/12$ then\\
  \hspace*{12mm} mark $h$\\
  \hspace*{12mm} if $x_{h+1}(t)+4\le x_{r(\mathcal{C},t)}(t)$ then\\
  \hspace*{16mm} $h\leftarrow h+1$, go to (1.4)\\
  \hspace*{12mm} else\\
  \hspace*{16mm} $h\leftarrow 1+\max\{i\,:\,i\in\mathcal{C}\}$\\
  \hspace*{16mm} if $h>n$ then STOP else go to (1.2) end if\\
  \hspace*{12mm} end if\\
  \hspace*{8mm} end if\\ 
  (1.6) choose $a,b\in\mathcal{C}$ with $x_h(t)+1\le x_a(t),x_b(t)\le x_h(t)+3$ and $x_a(t)<x_b(t)$ such that $x_a(t)=x_{b-1}(t)$\\
  (1.7) if $u<6k$ then\\
  \hspace*{12mm} $t\leftarrow t+1$\\
  \hspace*{12mm} $u\leftarrow n^\alpha+12$\\
  \hspace*{12mm} go to (1.2)\\
  \hspace*{8mm} end if\\
  (1.8) place $3k$ strategic agents at $x_a(t)-1$ and $x_b(t)+1$ each\\
  (1.9) $u\leftarrow u-6k$\\
  (1.10) if $x_{b}(t)+4\ge x_{r(\mathcal{C},t)}(t)$ then\\
  \hspace*{12mm} $h\leftarrow 1+\max\{i\,:\,i\in\mathcal{C}\}$\\
  \hspace*{12mm} if $h>n$ then STOP else go to (1.2) end if\\
  \hspace*{8mm} else\\
  \hspace*{12mm} $h\leftarrow \min\!\left\{i\in\mathcal{C}\,:\, x_i(t)>x_b(t)+4\right\}$\\
  \hspace*{12mm} go to (1.3)\\
  \hspace*{8mm} end if\\
  
\bigskip

\noindent
By $t$ we denote the current time, by $h$ we denote the index of a non-strategic agent that we do not have 
considered before, and by $u$ we denote the number of strategic agents that we can still use at time $t$. In (1.1) 
we initialize these three variables, i.e., we start at time $t=0$ with the first non-strategic agent $h=1$ and still can
use all $m=n^\alpha+12$ strategic agents.  
  
By $\mathcal{C}$ we denote the {\connectivity} component of agent $h$ at time $t$. Whenever we change $h$ to a non-strategic 
agent outside of $\mathcal{C}$ or $t$ is increased by one we go to (1.2) in the remaining steps of the algorithm.  

In (1.3) we check whether the so far unconsidered part of $\mathcal{C}$ is \textit{rather short}. If this is the case, we increase 
$h$ to a non-strategic agent of the \textit{next} {\connectivity} component. Here $h>n$ is used as a stopping criterion.

Reaching (1.4), we can assume $x_h(t)+4\le x_{r(\mathcal{C},t)}(t)$ and denote the number of non-strategic agents with 
an opinion at time $t$ from the interval $\left[x_h(t),x_h(t)+4\right]$ by $k$.

In (1.5) we check whether there are \textit{many}, compared to the number of strategic agents, non-strategic agents with an 
opinion in the mentioned subinterval of the opinion space of length $4$. If this is the case, we increase the index $h$.

Reaching (1.6), we can assume $k\le \overset{\ge 1}{\overbrace{\left(n^\alpha+12\right)/12}}$ and choose two non-strategic agents $a,b\in\mathcal{C}$ satisfying certain 
technical constraints. The existence of $a,b\in\mathcal{C}$ can be concluded from the fact that the non-strategic agents of 
$\mathcal{C}$ are connected, so that the interval $\left[x_h(t)+1,x_h(t)+3\right]$ of length $2$ contains at least 
two different opinions of non-strategic agents at time $t$.
   
In (1.7) it is checked if we can still use $6k$ strategic agents at time $t$. If not, the time is increased to the next time step.

Reaching (1.8), we can assume that we can still use $6k$ strategic agents at time $t$. (1.8) describes the precise placement 
of the strategic agents and (1.9) performs the bookkeeping of the number of used strategic agents.

In (1.10) we ensure that any subsequent placement of strategic agents in (1.8) will not interfere with non-strategic agents
treated so far.

\bigskip

\noindent    
After a finite number of time steps the algorithm of Step~(1) stops with $h>n$. There exists exactly one place in the algorithm where 
$t$ is increased by one, i.e., substep (1.7). Since here the condition $u<6k$ is satisfied, at least 
$(t-1)\cdot \frac{n^\alpha}{2}$ strategic agents have been placed in total. In (1.4) and (1.8) we ensure that each $6$ placed 
strategic agents correspond to a non-strategic agent without double counting. Thus, the algorithm of Step~(1) finishes 
after $O\!\left(n^{1-\alpha}\right)$ time steps. 
  
In (1.8) strategic agents are used to split {\connectivity} components. To be precise, we have    
$$
  x_a(t+1)\le \frac{3k\cdot\left(x_a(t)-1\right)+1\cdot x_a(t)+k_1\cdot \left(x_a(t)+1\right)}{3k+1+k_1}\le x_a(t)-\frac{1}{2},
$$  
where $k_1\le k-1$. Similarly, we deduce $x_b(t+1)\ge x_b(t)+\frac{1}{2}$, i.e., $x_b(t+1)-x_a(t+1)>1$. We remark that any 
placement of strategic agents in the subsequent operations does not clash with our previous placement of strategic agents, even 
if they are performed at the same time $t$. To be more precise, if two strategic agents are placed at the same time 
within the range $\left[x_{l(\mathcal{C},t)}(t),x_{r(\mathcal{C},t)}(t)\right]$ of the same {\connectivity} component, then their 
distance is either zero or at least $2$. If two strategic agents are placed at the same time in two different {\connectivity} components, 
then there do not exist non-strategic agents which are influenced by both. 

There are exactly two reasons why the algorithm of Step~(1) does not further split {\connectivity} components. Either they have a 
\textit{low width} or a \textit{high density}. If the condition in (1.5) is satisfied at time $t$, then at time $t+2$ (check (1.7) may cause 
a delay of one time step) agent~$h$ 
is contained in a {\connectivity} component with density, i.e., number of non-strategic agent in the {\connectivity} component 
divided by its length, at least $\frac{n^\alpha}{144}$. Note that the first and the last subinterval of length $4$ might 
not have a high density, while the remaining part, i.e., those agents where the condition in (1.5) applies, must indeed have a 
high density, since otherwise the {\connectivity} component would have been partitioned into several {\connectivity} components 
by the algorithm. Furthermore, those agents where the condition in (1.5) applies, except those of the first and the last 
subinterval of length $4$ are marked. 
It is easy to check that the remaining {\connectivity} components have a width of at most $8$. Thus at the time Step~(1) has 
been completed, each non-strategic agent is either contained in a {\connectivity} component of a marked agent or contained 
in a {\connectivity} component of width of at most $8$.
  
  \bigskip

  \noindent
  \textbf{Step (2):} Let $t$ be one time step after the completion of Step~(1). 
  Let $\mathcal{V}\subseteq N$ be the set of non-strategic agents which are either contained in a {\connectivity} 
  component with at least one marked agent or in a connected component $\mathcal{C}$ with $|\mathcal{C}|/w(\mathcal{C},t)\ge 
  n^\alpha/144$ at time $t$. Given a marked agent $v\in\mathcal{V}$ we know that at the time of labeling, $v$ was contained 
  in a {\connectivity} component with a density of at least $\frac{n^\alpha}{144}$. During the execution of Step~(1) such a 
  {\connectivity} component may have been split into several subcomponents. Some of these may have a \textit{high density} 
  others may not. In any case the summed widths of the subcomponents is not larger than the width of the original component. 
  Thus, if $\mathcal{V}$ is partitioned into {\connectivity} components $\mathcal{C}_1,\dots,\mathcal{C}_r$, then we have
  $
    \sum_{i=1}^r w(\mathcal{C}_i,t)\in O\!\left(n^{1-\alpha}\right)
  $.   
  
  At a given time $t'\ge t$ we choose a {\connectivity} component $\mathcal{C}\subseteq \mathcal{V}$ with width 
  larger than $1$ and place all $n^\alpha+12$ strategic agents at $x_{l(\mathcal{C},t')}(t')+1$. Since the number 
  of non-strategic agents at position 
  $x_{l(\mathcal{C},t')}(t')$ is at most $n$, this decreases the width of the corresponding {\connectivity} 
  component by at least $\frac{1}{2n^{1-\alpha}}$. Thus, 
  it takes at most $O\!\left(n^{2-2\alpha}\right)$ time steps until 
  all {\connectivity} components of agents in $\mathcal{V}$ have a width of at most $1$.

  \bigskip
  
  \noindent
  \textbf{Step (3):}
  Let $t$ be one time step after the completion of Step~(2). At time $t$ each connected component $\mathcal{C}$ has 
  a width of at most $8$ and a cardinality of at most $n^\alpha/2$, since $8/144<1/2$. Next we loop over all non-strategic 
  agents once.  
  Let $\mathcal{C}$ be the {\connectivity} component of our current agent $h$ at the current time 
  $t'$. If $w(\mathcal{C},t')=0$ we do nothing for $h$, i.e., we consider the next non-strategic agent. 
  If $0<w(\mathcal{C},t')\le 1$ 
  we again do nothing and have $w(\mathcal{C},t'+1)=0$ at time $t'+1$. In the remaining cases we have  
  $1<w(\mathcal{C},t')\le 8$ and $\left|\mathcal{C}\right|\le 
  \frac{n^\alpha}{2}$ and we place $\left|\mathcal{C}\right|$ strategic agents at $x_{l(\mathcal{C},t')}(t')+1$ 
  until the width 
  of the {\connectivity} component is at most $1$, which happens after at most $14$ time steps, since the length of the 
  interval is decreased by at least $\frac{1}{2}$ at every time step. We can clearly perform several 
  such operations at the same time as long as we have a sufficient number of strategic agents available. Since running out off available 
  strategic agents means that we have used at least half of them, Step~(3) can be done in $O\!\left(n^{1-\alpha}\right)$ time 
  steps. After the completion of Step~(3) all non-strategic agents are frozen.
\end{proof}

\begin{lemma}
  \label{lemma_lower_bound_alpha}
  For each $\alpha\in[0,1]$ we have $f(n,n^\alpha)\in \Omega\!\left(n^{\frac{2-2\alpha}{3}}\right)$.
\end{lemma}
\begin{proof}
  We set $k=\left\lfloor n^\beta\right\rfloor$, where $\beta=\frac{1-\alpha}{3}$, and consider the HK system 
  of $m=\left\lfloor n/k\right\rfloor$ dumbbell configurations consisting of $k-2$, $k-1$, or $k$ non-strategic agents 
  each (in the construction of a dumbbell configuration, we assume $k\equiv 1\pmod 3$). Of course, we also have to assume 
  that $k$ is sufficiently large. The possibly non-empty set of remaining non-strategic agents is placed far away from the 
  other $m$ {\connectivity} components at time $0$. If not affected by at least one strategic agent, the convergence 
  time of an arbitrary dumbbell {\connectivity} component is in $\Omega\!\left(n^{2\beta}\right)=\Omega\!\left(n^{\frac{2-2\alpha}{3}}\right)$. 
  Within $n^{2\beta}$ time steps, at most $n^{2\beta+\alpha}$ of the $m$ dumbbell {\connectivity} components can be affected 
  by at least a single strategic agent. Thus, we have $f(n,n^\alpha)\in \Omega\!\left(n^{\frac{2-2\alpha}{3}}\right)$.  
\end{proof}

We remark that for $\alpha=0$ 
Lemma~\ref{lemma_lower_bound_alpha} improves the lower bound 
from Corollary~\ref{cor_lower_bound} to $\Omega\!\left(n^{2/3}\right)$. If we replace the use of many dumbbell configurations 
by the same number of equidistant configurations an $\Omega\!\left(n^{\frac{1-\alpha}{2}}\right)$ lower bound can be concluded.

\section{Conclusion}
\label{sec_conclusion}
We have demonstrated that the convergence times of $\Omega(n)$ of the equidistant configuration and of $\Omega(n^2)$ of the 
dumbbell configuration can both be reduced to $O\!\left(n^{3/4}\right)$ using a single strategic agent. It would be very interesting 
to see optimized strategies and/or an improved analysis further reducing the upper bound of the optimal convergence time
for these two specific examples for $m=1$. Given a concrete HK system and a number $m$ of strategic agents, the optimal 
convergence time might be determined using an exact algorithm based on integer linear programming, similar to those ILP formulations 
presented in \cite{opinion_control,kurz2014long}. However, we do not go into details here and propose the development of 
(practically) efficient exact algorithms for the optimal control problem of minimizing the convergence time in the Hegselmann--Krause 
model as a research challenge.

We have further analyzed an example where the optimal convergence time using a single strategic agent is lower bounded by 
$\Omega\!\left(n^{2/3}\right)$. So, either this bound should be improved or the optimal convergence time of the two previously mentioned examples 
is better than $O\!\left(n^{3/4}\right)$. Since the upper bound for $f(n,1)$ is still $O(n^2)$, it would be nice if the considerably 
large gap could be narrowed. If enough strategic agents are available, any given HK system could be controlled in such a way that 
it converges in $2$ time steps. To be more precise, we have shown that $m=9n$ agents are sufficient. In some sense, the number 
of strategic agents must be at least as large as a constant fraction of the number of non-strategic ones in order to 
guarantee convergence in a constant number of time steps. So, we ask for the following tightening: Given a constant 
$c_2\in\mathbb{R}_{>0}$, what is the minimum constant $c_1\in\mathbb{R}_{>0}$, of course depending on $c_2$, such that we have 
$f(n,c_1n)\ge c_2$ for all sufficiently large $n$?

For the cases in between the two extreme situations of a single and that of very many strategic agents, we have proven a smooth upper 
bound meeting the best known bounds for the extreme situations up to a constant. However, we do not think that this general construction 
is tight. Indeed, there is a considerable gap to the presented lower bound. So, we ask for improvements.

Instead of distinguishing the stable state as the desired state, one can also ask for the minimal time needed to reach a consensus. However, 
the resulting problem is rather similar to the one studied in this paper.

Even without the presence of strategic agents there is a considerable gap between the best known upper $O(n^3)$ and lower bound
$\Omega(n^2)$ for convergence in the $1$-dimensional Hegselmann--Krause model. As demonstrated by the example analyzed in 
\cite{kurz2014long}, the estimation in the proof of the upper bound in \cite{bhattacharyya2013convergence} can be 
tight up to a constant for $\Omega(n)$ time steps. So, it seems that additional ideas are needed. 

All of our considerations were restricted to the one-dimensional case. The same questions can clearly be asked for higher dimensions 
and then for different norms.

There are variants of the HK model where it is still not known whether the system converges or not. One notable example is the heterogeneous 
version of a HK system, i.e., where the neighborhood radii of the agents are not necessarily the same. Here, one strategic 
agent is sufficient to guarantee convergence in a finite number of time steps, which follows from the same argument used in the proof of Lemma~\ref{lemma_improvement}. 
To be precise, the numerator $1$ of the lower bound in (2) has to be replaced by the smallest neighbourhood radius. Another example is the Hegselmann-Krause dynamics 
on the one-dimensional boundary of a circle assuming asymmetric influence ranges. For convergence results on the circle in the 
absence of strategic agents we refer the reader to \cite{hegarty2014circle}.

\section*{Acknowledgements} I thank Susanne B\"orner for carefully reading and commenting an earlier draft of this paper.


\providecommand{\href}[2]{#2}

\end{document}